\documentclass[12pt]{amsart}
\usepackage{amssymb}
\usepackage[shortlabels]{enumitem}
\usepackage[all]{xy}
\usepackage{xcolor}
\usepackage[margin=1in]{geometry} 
\usepackage[bookmarks, bookmarksdepth=2, colorlinks=true, linkcolor=blue, citecolor=blue, urlcolor=blue]{hyperref}
\usepackage{eucal}
\usepackage{contour}
\usepackage[normalem]{ulem}
\usepackage{dsfont}

\makeatletter
\@addtoreset{equation}{section}
\makeatother

\numberwithin{equation}{section}
\newtheorem{theorem}[equation]{Theorem}

\newtheorem{proposition}[equation]{Proposition}
\newtheorem{lemma}[equation]{Lemma}
\newtheorem{corollary}[equation]{Corollary}

\newtheorem{question}[equation]{Question}

\theoremstyle{definition}
\newtheorem{rmk}[equation]{Remark}
\newenvironment{remark}[1][]{\begin{rmk}[#1] \pushQED{\qed}}{\popQED \end{rmk}}
\newtheorem{eg}[equation]{Example}
\newenvironment{example}[1][]{\begin{eg}[#1] \pushQED{\qed}}{\popQED \end{eg}}
\newtheorem{defnaux}[equation]{Definition}
\newenvironment{definition}[1][]{\begin{defnaux}[#1]\pushQED{\qed}}{\popQED \end{defnaux}}

\newcommand{\bC}{\mathbf{C}}
\newcommand{\cC}{\mathcal{C}}

\newcommand{\cE}{\mathcal{E}}

\newcommand{\bF}{\mathbf{F}}

\newcommand{\rM}{\mathrm{M}}

\newcommand{\bN}{\mathbf{N}}

\newcommand{\bR}{\mathbf{R}}

\newcommand{\fS}{\mathfrak{S}}

\newcommand{\rf}{\mathrm{f}}

\newcommand{\arxiv}[1]{\href{http://arxiv.org/abs/#1}{{\tiny\tt arXiv:#1}}}

\newcommand{\DOI}[1]{\href{http://doi.org/#1}{\color{purple}{\tiny\tt DOI:#1}}}

\contourlength{1pt}

\contourlength{0.8pt}
\newcommand{\myuline}[1]{%
  \uline{\phantom{#1}}%
  \llap{\contour{white}{#1}}%
}
\DeclareMathOperator{\uRep}{\text{\myuline{\rm Rep}}}

\let\ul\underline
\renewcommand{\phi}{\varphi}
\DeclareMathOperator{\tr}{tr}
\DeclareMathOperator{\End}{End}
\DeclareMathOperator{\Aut}{Aut}
\DeclareMathOperator{\Hom}{Hom}
\newcommand{\id}{\mathrm{id}}
\renewcommand{\Vec}{\mathrm{Vec}}
\newcommand{\GL}{\mathbf{GL}}

\newcommand{\defn}[1]{\emph{#1}}
\newcommand{\bzero}{\mathbf{0}}
\newcommand{\bone}{\mathbf{1}}
\newcommand{\uotimes}{\mathbin{\ul{\otimes}}}

\title{Some fast-growing tensor categories}

\author{Andrew Snowden}

\date{May 27, 2023}

\begin{document}

\begin{abstract}
Given a semi-simple pre-Tannakian category over a finite field, we show that (a slight modification of) its linearization over a field of characteristic~0 is also semi-simple and pre-Tannakian. The key input is a result of Kuhn on the generic representation theory of finite fields. The resulting pre-Tannakian categories have substantially faster growth than previously known examples.
\end{abstract}

\maketitle
\tableofcontents

\section{Introduction}

\subsection{Growth in pre-Tannakian categories}

Pre-Tannakian categories are a natural class of tensor categories generalizing representation categories of algebraic groups (see \S \ref{ss:tencat} for the definition). Given an object $X$ in a pre-Tannakian category $\cC$ and a non-negative integer $n$, we let $\phi_X(n)$ be the (Jordan--H\"older) length of $X^{\otimes n}$. We refer to the function $\phi_X$ as the \defn{profile} of $X$. This paper concerns the following question, brought to our attention by Pavel Etingof:

\begin{question}
How fast can the function $\phi_X$ grow?
\end{question}

This question is a useful barometer for our understanding of pre-Tannakian categories. If there is some constraint on the growth of $\phi_X$, that would be a fundamental aspect of pre-Tannakian categories that is currently unknown. If not, we should be able to construct pre-Tannakian categories with arbitrary growth, but this seems to be currently out of reach. In this note we construct new pre-Tannakian categories with substantially faster growth than previously known examples.

Before stating our results, we recall what is known. If $\cC$ is the category of finite dimensional representations of an algebraic (super)group, then the length of $X^{\otimes n}$ is certainly at most the dimension of the underlying vector space, and so $\phi_X$ has at most exponential growth. Conversely, Deligne \cite{Deligne1} proved that if $\cC$ is a pre-Tannakian category in characteristic~0 such that $\phi_X$ has at most exponential growth for all objects $X$ then $\cC$ is super Tannakian.

Deligne--Milne \cite[Example~1.27]{DeligneMilne} gave the first example of a pre-Tannakian category with superexponential growth: the category $\uRep(\GL_t)$ has an object $X$ with $\phi_X(n) \approx \exp(n \log{n})$\footnote{The approximations for $\phi_X$ in the introduction are not asymptotic relations.}. Deligne \cite{Deligne2} later introduced the category $\uRep(\fS_t)$, which has similar growth. Knop \cite{Knop, Knop2} put these constructions into a general framework, and constructed many more pre-Tannakian categories, including $\uRep(\GL_t(\bF_q))$. This category has an object $X$ with $\phi_X(n) \approx \exp(n^2)$. Prior to this paper, Knop's category held the record for fastest growth.

In this paper, we prove the following theorem:

\begin{theorem} \label{mainthm1}
There is a semi-simple pre-Tannakian category $\cC$ over a field of characteristic~0 admitting an object $X$ such that $\phi_X(n) \approx \exp(\exp(n \log{n}))$.
\end{theorem}

See Example~\ref{ex:line} for more precise information on $\phi_X$.

\begin{remark}
Sophie Kriz \cite{Kriz} recently constructed non-abelian rigid tensor categories of arbitrary growth. These categories have nilpotents with non-zero trace, and therefore do not admit pre-Tannakian envelopes.
\end{remark}

\subsection{The construction}

Let $\cE$ be a semi-simple pre-Tannakian category over a finite field $\bF$. Fix an algebraically closed field $k$ of characteristic~0, and let $k[\cE]^{\sharp}$ be the additive--Karoubi envelope of the $k$-linearization of $\cE$. The monoid $(\bF, \cdot)$ acts on the category $k[\cE]^{\sharp}$, and so the monoid algebra $k[\bF]$ does as well. For a character $\chi \colon \bF^{\times} \to k^{\times}$, we define an idempotent $e_{\chi}$ in $k[\bF]$, and let $k[\cE]^{\sharp}_{\chi}$ be the full subcategory of $k[\cE]^{\sharp}$ spanned by objects $X$ satisfying $X=e_{\chi} X$. The tensor product $\otimes$ on $\cE$ induces a symmetric monoidal structure on $k[\cE]^{\sharp}_{\chi}$, which we denote $\uotimes$. We prove:

\begin{theorem} \label{mainthm2}
The category $k[\cE]^{\sharp}_{\chi}$ is semi-simple and pre-Tannakian.
\end{theorem}

The proof of this theorem crucially relies on results of Kuhn \cite{Kuhn2} (which in turn builds on earlier work of Kov\'acs \cite{Kovacs}) on the representation theory of the monoid of matrices over a finite field. If $X$ is an object of $\cE$ then there is an associated object $Y=e_{\chi} [X]$ in $k[\cE]^{\sharp}_{\chi}$, and we have $\phi_Y \approx \exp(\phi_X)$. Thus by taking $\cE=\Vec_{\bF}^{\rf}$ to be the category of finite dimensional $\bF$-vector spaces, we obtain a doubly exponential profile (see Example~\ref{ex:vec}). By taking $\cE$ to be the Delannoy category (studied in \cite{line}), we obtain the even faster growth appearing in Theorem~\ref{mainthm1} (see Example~\ref{ex:line}).

\subsection{Additional comments}

We make a few additional remarks:
\begin{enumerate}
\item In recent work \cite{repst}, we constructed pre-Tannakian categories associated to oligomorphic groups. The categories constructed in Examples~\ref{ex:vec} and~\ref{ex:line} below can be obtained from this construction (as the main result of \cite{discrete} implies), and this is actually how we first found them. We plan to revisit this in future work.
\item If $\cE$ is a tensor category then $k[\cE]$ carries two tensor product operations $\ul{\oplus}$ and $\uotimes$, induced by the direct sum and tensor product on $\cE$. In this paper, we only use $\uotimes$. In prior work where categories like $k[\cE]$ appear, the tensor product $\ul{\oplus}$ (or similar ones) has played an important role, and $\uotimes$ seems to have not received much attention. See, e.g., \cite[\S 3.5]{Kuhn1} or \cite[\S 1.1]{Nagpal}.
\item If $\cE$ is any symmetric monoidal category then $k[\cE]^{\sharp}$ is naturally a tensor category. An interesting example is when $\cE$ is the category of ``vector spaces'' over the Boolean semiring. This is closely related to the ``correspondence functors'' studied by Bouc--Th\'evenaz \cite{Bouc}, and we plan to return to this example in future work. Unfortunately, it is not clear how to find pre-Tannakian categories from $k[\cE]^{\sharp}$ in general.
\item Let $\cE=\Vec_{\bF}^{\rf}$ be the category of finite dimensional vector spaces over a finite field $\bF$, so that $k[\cE]^{\sharp}_{\chi}$ is pre-Tannakian by Theorem~\ref{mainthm2}. This is the simplest construction of a pre-Tannakian category of superexponential growth that we know. It is also the first example we know of a pre-Tannakian category that does not admit deformations yet has objects of non-real dimension, assuming $k=\bC$ and $\chi$ is non-real.
\end{enumerate}

\subsection{Notation}

We fix an algebraically closed field $k$ of characteristic~0 for the duration of the paper. Some other important notation:
\begin{description}[align=right,labelwidth=2.5cm,leftmargin=!]
\item[ $\bF$ ] A finite field.
\item[ $\cE$ ] A category (typically with finite $\Hom$ sets).
\item[ $\cC$ ] A $k$-linear category.
\item[ $\Vec_F$] The category of $F$-vector spaces.
\item[ $\Vec^{\rf}_F$] The category of finite dimensional $F$-vector spaces.
\end{description}

\subsection*{Acknowledgments}

We thank Pavel Etingof, Nate Harman, Sophie Kriz, Noah Snyder, and Peter Webb for helpful conversations.

\section{Categorical preliminaries}

\subsection{Tensor categories} \label{ss:tencat}

We now recall some basic concepts related to tensor categories.

\begin{definition}
Let $\cC$ be a symmetric monoidal category with unit object $\bone$. A \defn{dual} of an object $X$ is an object $X^{\vee}$ equipped with maps $\alpha \colon \bone \to X \otimes X^{\vee}$ and $\beta \colon X^{\vee} \otimes X \to \bone$ such that the composition
\begin{displaymath}
\xymatrix@C=4em{
X \ar[r]^-{\alpha \otimes \id} & X \otimes X^{\vee} \otimes X \ar[r]^-{\id \otimes \beta} & X }
\end{displaymath}
is the identity, and similarly for the analogous sequence starting with $X^{\vee}$. If $X$ has a dual then it is unique up to isomorphism, and $X$ is called \defn{rigid}. The category $\cC$ is \defn{rigid} if every object is. See \cite[\S 2.10]{EGNO} for details.
\end{definition}

We note that a symmetric monoidal functor maps rigid objects to rigid objects (see \cite[Exercise~2.10.6]{EGNO}). Fix a field $F$ for the next two definitions.

\begin{definition} \label{def:tensor}
An \defn{$F$-linear tensor category} is an $F$-linear additive category equipped with an $F$-bilinear symmetric monoid structure.
\end{definition}

\begin{definition} \label{def:pretan}
An $F$-linear tensor category $\cC$ is \defn{pre-Tannakian} if:
\begin{enumerate}
\item $\cC$ is abelian and all objects have finite length.
\item All $\Hom$ spaces in $\cC$ are finite dimensional over $F$.
\item $\cC$ is rigid.
\item $\End_{\cC}(\bone)=F$, where $\bone$ is the monoidal unit. \qedhere
\end{enumerate}
\end{definition}

We warn the reader that \cite{EGNO} uses the term ``tensor category'' in a much stronger sense than us: ``pre-Tannakian'' in our sense is equivalent to ``symmetric tensor category'' in the sense of \cite{EGNO}. Our use of ``pre-Tannakian'' follows \cite[\S 2.1]{ComesOstrik}.

\subsection{Linearization} \label{ss:linear}

Let $\cE$ be a category. We define the \defn{linearization} of $\cE$, denoted $k[\cE]$, to be the following $k$-linear category. The objects of $k[\cE]$ are the same as the objects of $\cE$. For clarity, if $X$ is an object of $\cE$ then we write $[X]$ when we regard it as an object of $k[\cE]$. The morphisms in $k[\cE]$ are defined by
\begin{displaymath}
\Hom_{k[\cE]}([X],[Y]) = k[\Hom_{\cE}(X,Y)],
\end{displaymath}
where on the right side $k[S]$ means the $k$-vector space with basis $S$. For a morphism $\phi \colon X \to Y$ in $\cE$, we write $[\phi]$ for the corresponding morphism in $k[\cE]$. Composition in $k[\cE]$ is defined in the obvious manner: it is $k$-bilinear and satisfies $[\psi] \circ [\phi]=[\psi \circ \phi]$ for composable morphisms $\psi$ and $\phi$ of $\cE$. There is a natural functor $\cE \to k[\cE]$ defined by $X \mapsto [X]$ and $\phi \mapsto [\phi]$.

Suppose now that $\cE$ has a symmetric monoidal structure, and write $\otimes$ for the monoial operation. Then $k[\cE]$ inherits a $k$-bilinear symmetric monoidal structure. We denote the monoidal operation by $\uotimes$. On objects it is given by
\begin{displaymath}
[X] \uotimes [Y] = [X \otimes Y],
\end{displaymath}
and the formula for morphisms is similar. The functor $\cE \to k[\cE]$ is symmetric monoidal. In particular, if $X$ is a rigid object of $\cE$ then $[X]$ is a rigid object of $k[\cE]$.

\subsection{Additive and Karoubi envelopes} \label{ss:envelope}

Let $\cC$ be a $k$-linear category. The \defn{additive envelope} of $\cC$, denoted $\cC^{\oplus}$, is the $k$-linear category whose objects are tuples $(X_1, \ldots, X_n)$ where $X_1, \ldots, X_n$ are objects of $\cC$, and where morphisms $(X_1, \cdots, X_n) \to (Y_1, \cdots, Y_m)$ are given by $n \times m$ matrices with entries in $\Hom_{\cC}(X_i, Y_j)$. Composition is defined by matrix multiplication. The category $\cC^{\oplus}$ is additive, and $(X_1, \ldots, X_n)$ is the direct sum of the objects $X_1, \ldots, X_n$; we therefore prefer to use the notation $X_1 \oplus \cdots \oplus X_n$. Note that if $X$ and $Y$ are objects of $\cC$ such that $X \oplus Y$ exists in $\cC$, then $(X, Y)$ and $(X \oplus Y)$ are isomorphic in $\cC^{\oplus}$. We refer to \cite[\S 2.5]{ComesWilson} for details.

The \defn{Karoubi envelope} of $\cC$, denoted $\cC^{\rm kar}$, is the $k$-linear category whose objects are pairs $(X, e)$, where $X$ is an object of $\cC$ and $e$ is an idempotent of $\End_{\cC}(X)$, and where
\begin{displaymath}
\Hom_{\cC^{\rm kar}}((X,e), (Y,f))=\{ \phi \in \Hom_{\cC}(X,Y) \mid f \circ \phi = \phi = \phi \circ e \}.
\end{displaymath}
We refer to \cite[\S 2.6]{ComesWilson} for details. We typically write $eX$ for the object $(X,e)$ of $\cC^{\rm kar}$, and omit $e$ when $e=1$.

The \defn{additive--Karoubi envelope} of $\cC$, denoted $\cC^{\sharp}$, is $(\cC^{\oplus})^{\rm kar}$. This category is $k$-linear, additive, and Karoubian (all idempotents split).

Suppose now that $\cC$ carries a $k$-bilinear symmetric monoidal structure. One easily sees that this induces such a structure on $\cC^{\oplus}$, $\cC^{\rm kar}$, and $\cC^{\sharp}$. In this way, $\cC^{\oplus}$ and $\cC^{\sharp}$ are tensor categories (in the sense of Definition~\ref{def:tensor}). The natural functors from $\cC$ to these categories are symmetric monoidal. Thus if $X$ is a rigid object of $\cC$ then it is also a rigid object in these categories. Since direct sums and summands of rigid objects are rigid, it follows that if $\cC$ is a rigid category then each of $\cC^{\oplus}$, $\cC^{\rm kar}$, and $\cC^{\sharp}$ are rigid as well.

\subsection{Semisimplicity}

We require the following criterion for semisimplicity of Karoubi envelopes; this is certainly well-known, but for the sake of completeness we include a proof.

\begin{proposition} \label{prop:kar}
Let $\cC$ be a $k$-linear additive category in which all endomorphism algebras are finite dimensional and semi-simple. Then the Karoubi envelope $\cC^{\rm kar}$ is a semi-simple abelian category.
\end{proposition}

\begin{proof}
We first claim that all endomorphism algebras in $\cC^{\rm kar}$ are semi-simple. Thus let $X$ be a given object of $\cC^{\rm kar}$. By definition, there is an object $Y$ of $\cC^{\rm kar}$ such that $X \oplus Y$ is isomorphic to an object of $\cC$, and so $R=\End(X \oplus Y)$ is semi-simple. If $e \in R$ is the projector onto $X$ then $\End(X)=eRe$, and is thus semi-simple; note that $eRe=\End_R(Re)$ and $Re$ is a semi-simple $R$-module.

Since endomorphism algebras in $\cC^{\rm kar}$ are finite dimensional, every object of $\cC^{\rm kar}$ decomposes into a finite direct sum of indecomposable objects. If $X$ is indecomposable then $\End(X)$ is a finite dimensional semi-simple algebra with no idempotents, and thus $\End(X)=k$. In particular, a non-zero endomorphism of $X$ is an isomorphism.

Let $\phi \colon X \to Y$ be a non-zero map of indecomposable objects. We claim that $\phi$ is an isomorphism. Since $R=\End(X \oplus Y)$ is semi-simple, there is an element that has non-zero trace pairing with $\phi$. Since $\Hom(X,Y)$ is orthogonal to $\End(X)$ and $\End(Y)$ under the trace pairing, it follows that there is some $\psi \colon Y \to X$ such that $\tr(\phi \psi) \ne 0$. Since the trace pairing is symmetrical, it follows that $\tr(\psi \phi) \ne 0$ too. Thus $\phi \psi$ and $\psi \phi$ are non-zero endomorphisms of $X$ and $Y$, and therefore isomorphisms. Thus $\phi$ is an isomorphism, as claimed.

It now follows that $\cC^{\rm kar}$ is semi-simple. Its simple objects are exactly its indecomposable objects.
\end{proof}

\section{Main results}

\subsection{Kuhn's theorem}

Fix a finite field $\bF$ and let $\cE=\Vec^{\rf}_{\bF}$ be the category of finite dimensional $\bF$-vector spaces. For non-negative integers $n_1, \ldots, n_r$, put
\begin{displaymath}
R_{n_1,\ldots,n_r}(\bF) = \End_{k[\cE]^{\oplus}}([\bF^{n_1}] \oplus \cdots \oplus [\bF^{n_r}]).
\end{displaymath}
This is a finite dimensional $k$-algebra.

We now give a direct description of this algebra. For non-negative integers $n$ and $m$, let $\rM_{n,m}(\bF)$ be the set of $n \times m$ matrices with entries in $\bF$. We have a map
\begin{displaymath}
\rM_{n,m}(\bF) \times \rM_{m,\ell}(\bF) \to \rM_{n,\ell}(\bF)
\end{displaymath}
given by matrix multiplication, which leads to a $k$-linear map
\begin{displaymath}
k[\rM_{n,m}(\bF)] \otimes k[\rM_{m,\ell}(\bF)] \to k[\rM_{n,\ell}(\bF)],
\end{displaymath}
We can describe $R_{n_1, \ldots, n_r}(\bF)$ as the space of $r \times r$ matrices whose $(i,j)$ entry belongs to $k[\rM_{n_i,n_j}(\bF)]$. The multiplication is matrix multiplication, using the above map to multiply individual entries.

The following is the key input needed for our main results.

\begin{theorem}[Kuhn] \label{thm:kuhn}
The algebra $R_{n_1,\ldots,n_r}(\bF)$ is semi-simple.
\end{theorem}

\begin{proof}
Let $\cC$ be the category of all functors $\cE \to \Vec_k$. By \cite[Corollary~1.3]{Kuhn2}, this is a semi-simple Grothendieck abelian category. Thus if $M$ is an object of $\cC$ such that $\End_{\cC}(M)$ is finite dimensional then $M$ has finite length and $\End_{\cC}(M)$ is a semi-simple algebra. Let $P_n \colon \cE \to \Vec_k$ be the functor defined by $P_n(V)=k[\Hom_{\cE}(\bF^n, V)]$. By Yoneda's lemma, we have
\begin{displaymath}
\Hom_{\cC}(P_n,P_m)=k[\Hom_{\cE}(\bF^m, \bF^n)]
\end{displaymath}
(see also \cite[\S 2.2]{Kuhn2}), and so
\begin{displaymath}
R_{n_1,\ldots,n_r}(\bF) = \End_{\cC}(P_{n_1} \oplus \cdots \oplus P_{n_r}).
\end{displaymath}
Thus $R_{n_1,\ldots,n_r}(\bF)$ is semi-simple.
\end{proof}

\begin{remark}
When $r=1$, the ring $R_n(\bF)$ is simply the monoid algebra of $\rM_{n,n}(\bF)$. See \cite{Kovacs} or \cite[Corollary~2.7]{Kuhn2} for the precise structure of this algebra.
\end{remark}

\subsection{Linearizing semi-simple categories}

Fix a semi-simple abelian category $\cE$ in which all $\Hom$ spaces are finite abelian groups. We now investigate the structure of $k[\cE]^{\sharp}$. The following is our main result in this direction.

\begin{theorem} \label{thm:lin-ss}
The category $k[\cE]^{\sharp}$ is a semi-simple abelian category.
\end{theorem}

We prove the theorem after some preparatory material. A \defn{generalized matrix algebra} of order $n$ is a finite dimensional $k$-algebra $A$ endowed with a decomposition $A=\bigoplus_{1 \le i,j \le n} A_{i,j}$, as a $k$-vector space, such that $A_{i,j} A_{j,\ell} \subset A_{i,\ell}$ and $A_{i,j} A_{r,s}=0$ if $j \ne r$. Suppose that $A$ and $B$ are two generalized matrix algebras of the same order $n$. We define the \defn{Segre product} of $A$ and $B$, denoted $A \boxtimes B$, to be the generalized matrix algebra of order $n$ with
\begin{displaymath}
(A \boxtimes B)_{i,j} = A_{i,j} \otimes B_{i,j},
\end{displaymath}
and the obvious multiplication law.

\begin{lemma} \label{lem:lin-ss-1}
If $A$ and $B$ as above are both semi-simple then so is $A \boxtimes B$.
\end{lemma}

\begin{proof}
Let $a$ be the identity element of $A$. This only has diagonal entries, meaning $a=\sum_{i=1}^n a_i$ with $a_i \in A_{i,i}$; moreover, the $a_i$'s are orthogonal idempotents. Similarly $b=\sum_{i=1}^n b_i$, where $b$ is the identity of $B$. Let $C=A \otimes B$, which is a semi-simple algebra, and let $e=\sum_{1 \le i \le n} a_i \otimes b_i$, which is an idempotent. One easily sees that $A \boxtimes B$ is identified with $eCe=\End_C(Ce)$, and is thus semi-simple.
\end{proof}

\begin{lemma} \label{lem:lin-ss-2}
Let $X=[X_1] \oplus \cdots \oplus [X_n]$ be an object of $k[\cE]^{\oplus}$. Suppose we have a decomposition $X_i=Y_i \oplus Z_i$ in $\cE$ for each $i$, such that $Y_i$ and $Z_j$ have no common simple constituents for all $i$ and $j$. Let $Y=[Y_1] \oplus \cdots \oplus [Y_n]$ and $Z=[Z_1] \oplus \cdots \oplus [Z_n]$. Then $\End_{k[\cE]^{\oplus}}(X)$ is the Segre product of $\End_{k[\cE]^{\oplus}}(Y)$ and $\End_{k[\cE]^{\oplus}}(Z)$.
\end{lemma}

\begin{proof}
Let $A$, $B$, and $C$ be the endomorphism algebras of $X$, $Y$, and $Z$ in $k[\cE]^{\oplus}$. We note that each of these are naturally generalized matrix algebras of order $n$, and so it makes sense to speak of the Segre product of $B$ and $C$. We have
\begin{align*}
A_{i,j} &= \Hom_{k[\cE]}([X_i], [X_j])
= k[\Hom_{\cE}(X_i,X_j)] \\
&= k[\Hom_{\cE}(Y_i,Y_j) \oplus \Hom_{\cE}(Z_i,Z_j)]
= B_{i,j} \otimes C_{i,j}
\end{align*}
In the third step we used that $\Hom_{\cE}(Y_i,Z_j)$ and $\Hom_{\cE}(Z_i,Y_j)$ vanish, which follows since $Y_i$ and $Z_j$ have no common constituents. We thus see that $A$ is identified with $B \boxtimes C$, and one easily verifies that this identification is compatible with the algebra structures.
\end{proof}

\begin{lemma} \label{lem:lin-ss-3}
Let $X=[X_1] \oplus \cdots \oplus [X_n]$ be an object of $k[\cE]^{\oplus}$. Suppose there is a simple object $L$ of $\cE$ such that each $X_i$ is $L$-isotypic; say $X_i \cong L^{\oplus m_i}$. Let $\bF=\End_{\cE}(L)$. Then $\End_{k[\cE]^{\oplus}}(X)$ is isomorphic to $R_{m_1,\ldots,m_n}(\bF)$.
\end{lemma}

\begin{proof}
This is essentially the definition of $R_{m_1,\ldots,m_n}(\bF)$: note that the full subcategory of $\cE$ spanned by $L$-isotypic objects is equivalent to $\Vec_{\bF}^{\rf}$.
\end{proof}

\begin{remark}
The ring $\bF$ in Lemma~\ref{lem:lin-ss-3} is a finite field by Wedderburn's ``little theorem.''
\end{remark}

\begin{proof}[Proof of Theorem~\ref{thm:lin-ss}]
Let $X$ be an object of $k[\cE]^{\oplus}$ and put $R=\End_{k[\cE]^{\oplus}}(X)$. Write $X=[X_1] \oplus \cdots \oplus [X_n]$ for objects $X_1, \ldots, X_n$ of $\cE$. Let $L_1, \ldots, L_m$ be the simple objects of $\cE$ appearing in $X_1, \ldots, X_n$, and let $X_i(j)$ be the $L_j$-isotypic piece of $X_i$. Put $X(j)=[X_1(j)] \oplus \cdots \oplus [X_n(j)]$ and $R(j)=\End_{k[\cE]^{\oplus}}(X(j))$. By Theorem~\ref{thm:kuhn} and Lemma~\ref{lem:lin-ss-3}, the algebra $R(j)$ is semi-simple. By Lemma~\ref{lem:lin-ss-2}, we have $R=R(1) \boxtimes \cdots \boxtimes R(m)$. Thus $R$ is semi-simple by Lemma~\ref{lem:lin-ss-1}. We thus see that all endomorphism algebras in $k[\cE]^{\oplus}$ are semi-simple, and so $k[\cE]^{\sharp}$ is semi-simple by Proposition~\ref{prop:kar}.
\end{proof}

\begin{remark}
Let $\{L_i\}_{i \in I}$ be representatives for the isomorphism classes of simple objects in $\cE$, and put $\bF_i=\End_{\cE}(L_i)$. One can show, using \cite[Theorem~1.1]{Kuhn2}, that the simple objects of $k[\cE]^{\sharp}$ correspond to tuples $(\rho_i)_{i \in I}$ where $\rho_i$ is an irreducible $k$-representation of $\GL_{n(i)}(\bF_i)$, for some $n(i) \in \bN$, such that $n(i)=0$ for all but finitely many $i$. It does not seem obvious how to describe the tensor product of simples in $k[\cE]^{\sharp}$.
\end{remark}

\subsection{Some idempotents}

Suppose now that $\cE$ is an $\bF$-linear semi-simple abelian category, for some finite field $\bF$, and that all $\Hom$ spaces in $\cE$ are finite dimensional. Theorem~\ref{thm:lin-ss} implies that $k[\cE]^{\sharp}$ is semi-simple. This category has a natural $k[\bF]$-linear structure, where $k[\bF]=R_1(\bF)$ is the monoid algebra of $\bF$ under multiplication. We now examine how to decompose $k[\cE]^{\sharp}$ using idempotents of this algebra.

For $a \in \bF$, let $[a]$ be the corresponding element of the monoid algebra $k[\bF]$. Let $e_0=[0]$, and for a character $\chi \colon \bF^{\times} \to k^{\times}$ put
\begin{displaymath}
e_{\chi}=-\delta_{\chi}[0]+\frac{1}{q-1} \sum_{a \in \bF^{\times}} \chi^{-1}(a) [a]
\end{displaymath}
where $\delta_{\chi}$ is~1 if $\chi$ is the trivial character, and~0 otherwise. One readily verifies that $e_0$ and the $e_{\chi}$ are the primitive idempotents of $k[\bF]$.

If $X$ is an object of $k[\cE]^{\sharp}$, then we have a decomposition
\begin{displaymath}
X = e_0 X \oplus \bigoplus_{\chi} e_{\chi} X.
\end{displaymath}
Define $k[\cE]^{\sharp}_{\chi}$ to be the full subcategory of $k[\cE]^{\sharp}$ on objects $X$ such that $X=e_{\chi} X$, and similarly define $k[\cE]^{\sharp}_0$. These are each $k$-linear semi-simple abelian categories. The above formula gives a decomposition
\begin{equation} \label{eq:decomp}
k[\cE]^{\sharp} = k[\cE]^{\sharp}_0 \oplus \bigoplus_{\chi} k[\cE]^{\sharp}_{\chi}.
\end{equation}
Let $\bzero$ be the zero object of $\cE$. We have $\End_{k[\cE]^{\sharp}}([\bzero])=k$, and so $[\bzero]$ is a simple object of $k[\cE]^{\sharp}$. One can show that $k[\cE]^{\sharp}_0$ consists of the $[\bzero]$-isotypic objects.

We now examine the $\Hom$ spaces in $k[\cE]^{\sharp}_{\chi}$. Let $X$ and $Y$ be objects of $k[\cE]^{\sharp}$. We let $k[\bF]$ act on $\Hom_{k[\cE]^{\sharp}}(X,Y)$ through its action on $Y$. A morphism $\phi \colon X \to Y$ maps into $e_{\chi} Y$ if and only if it satisfies $e_{\chi} \phi=\phi$, and it will then factor through $e_{\chi} X$. We thus find
\begin{equation} \label{eq:hom}
\Hom_{k[\cE]^{\sharp}_{\chi}}(e_{\chi} X, e_{\chi} Y) = e_{\chi} \Hom_{k[\cE]^{\sharp}}(X, Y).
\end{equation}
Recall that for a non-negaitve integer $n$, the corresponding \defn{$q$-integer} is defined by 
\begin{displaymath}
[n]_q = \frac{q^n-1}{q-1}.
\end{displaymath}
The following is the main result we are after:

\begin{proposition} \label{prop:Hom-chi}
Let $X$ be an object of $\cE$ and let $Y=e_{\chi} [X]$ be the corresponding object of $\cC=k[\cE]^{\sharp}_{\chi}$. Then $\dim_k \End_{\cC}(Y) = [\dim_{\bF} \End_{\cE}(X)]_q$ where $q=\# \bF$.
\end{proposition}

\begin{proof}
First suppose that $V$ is an $\bF$-vector space of dimension $d$, and regard $k[V]$ as a $k[\bF]$-module in the natural manner. The space $(1-e_0) k[V]$ has dimension $q^d-1$, with basis $[v]-[0]$ with $v \in V$ non-zero. The group $\bF^{\times}$ acts freely on $V \setminus \{0\}$, and so $(1-e_0) k[V]$ is free as an $k[\bF^{\times}]$-module. Identifying $k[\bF^{\times}]$ with $k[\bF]/(e_0)$, we find that the dimension of $e_{\chi} k[V]$ is independent of $\chi$. Since there are $q-1$ choices for $\chi$, we see that $e_{\chi} k[V]$ has dimension $[q]_d$.

Now, we have
\begin{displaymath}
\End_{\cC}(Y)=e_{\chi} \End_{k[\cE]}([X]) = e_{\chi} k[\End_{\cE}(X)],
\end{displaymath}
where in the first step we used \eqref{eq:hom}. The action of $k[\bF]$ on $k[\End_{\cE}(X)]$ is induced from the $\bF$-linear sturcture on $\End_{\cE}(X)$. Thus the result follows from the first paragraph.
\end{proof}

\subsection{Linearizing pre-Tannakian categories}

Fix a finite field $\bF$, an $\bF$-linear semi-simple pre-Tannakian category $\cE$, and a character $\chi \colon \bF^{\times} \to k^{\times}$. Let $\cC$ be category $k[\cE]^{\sharp}_{\chi}$.

\begin{theorem} \label{thm:lin-pt}
$\cC$ is a $k$-linear semi-simple pre-Tannakian category.
\end{theorem}

\begin{proof}
We first clarify the tensor structure on $\cC$. Let $X$ and $Y$ be objects of $ k[\cE]^{\sharp}$. One readily verifies the identifications
\begin{equation} \label{eq:echi}
e_{\chi} (X \uotimes Y)=(e_{\chi} X) \uotimes Y = X \uotimes (e_{\chi} Y) = (e_{\chi} X) \uotimes (e_{\chi} Y).
\end{equation}
The analogous formula with $e_0$ holds as well. This shows that each piece in the decomposition \eqref{eq:decomp} (and in particular $\cC$) is closed under the tensor product, and that different pieces of the decomposition tensor to~0. Moreover, if $\bone$ is the monoidal unit in $\cE$ then $[\bone]$ is the unit in $k[\cE]^{\sharp}$, and the above formula shows that $e_{\chi} [\bone]$ is the unit in $\cC$.

We now show that $\cC$ satisfies the conditions of Definition~\ref{def:pretan}. By Theorem~\ref{thm:lin-ss}, $\cC$ is a semi-simple abelian category, which verifies (a). It is clear that all $\Hom$ spaces in $\cC$ are finite dimensional over $k$, and so (b) holds. Since $\cE$ is rigid, so is $k[\cE]$ (see the final sentence of \S \ref{ss:linear}), and therefore so is $k[\cE]^{\sharp}$ (see the final paragraph of \S \ref{ss:envelope}). From \eqref{eq:echi}, it follows that the functor $k[\cE]^{\sharp} \to \cC$ given by $X \mapsto e_{\chi} X$ is symmetric monoidal, and so $\cC$ is rigid. Thus (c) holds. Finally, $\End_{\cC}(e_{\chi} [\bone])$ is 1-dimensional over $k$ by Proposition~\ref{prop:Hom-chi}, and so (d) holds.
\end{proof}

Recall that for an object $X$ of a pre-Tannakian category, $\phi_X(n)$ is the length of $X^{\otimes n}$. We define $\psi_X(n)$ to be the dimension of $\End(X^{\otimes n})$. In the semi-simple case these two quantities are related by the inequalities
\begin{displaymath}
\phi_X(n) \le \psi_X(n) \le \phi_X(n)^2.
\end{displaymath}
The following is the main result we are after concerning growth rates in $\cC$.

\begin{proposition} \label{prop:profile}
Let $X$ be an object of $\cE$ and let $Y=e_{\chi} [X]$ be the corresponding object of $\cC$. Then $\psi_Y(n) = [\psi_X(n)]_q$ where $q=\# \bF$.
\end{proposition}

\begin{proof}
We have $Y^{\uotimes n}=e_{\chi} [X^{\otimes n}]$, and so the result follows from Proposition~\ref{prop:Hom-chi}.
\end{proof}

We can now give some explicit examples of high growth rates.

\begin{example} \label{ex:vec}
Let $\cE=\Vec_{\bF}^{\rf}$. Suppose that $X$ is a $d$-dimensional vector space over $\bF$, with $d \ge 2$. Then $\psi_X(n)=d^{2n}$. Thus letting $Y=e_{\chi} [X]$, we find $\psi_Y(n)=[d^{2n}]_q$, and so we obtain a bound $\phi_Y(n) \ge A \exp(B \exp(C n))$ for explicit positive real numbers $A$, $B$, and $C$. We thus obtain a doubly exponential profile.
\end{example}

\begin{example}
Let $p$ be a prime number and let $\cE=\mathrm{Ver}_p$ be the Verlinde category, which is a semi-simple pre-Tannakian category over the field $\bF$ with $p$ elements (see \cite[\S 3.2]{Ostrik}). This is similar to the previous case, and leads to a doubly exponential profile.
\end{example}

\begin{example} \label{ex:line}
Let $G=\Aut(\bR,<)$ be the group of order preserving self-bijections of the real line, which is an oligomorphic group. In \cite[\S 17]{repst}, we constructed an $F$-linear semi-simple pre-Tannakian category $\uRep_F(G)$ associated to $G$, for any field $F$. This was the first example of a semi-simple pre-Tannakian category of superexponential growth in positive characteristic. It was studied in detail in \cite{line}, where it was named the Delannoy category.

Fix a finite field $\bF$, and let $\cE=\uRep_{\bF}(G)$. The category $\cE$ contains a basic representation $\cC(\bR)$ (the space of $\bF$-valued ``Schwartz functions'' on $\bR$), and we have $\cC(\bR)^{\otimes n}=\cC(\bR^n)$. The endomorphism algebra of this object has an $\bF$-basis corresponding to $G$-orbits on $\bR^{2n}$. The number $a(n)$ of $G$-orbits on $\bR^n$ is the \defn{ordered Bell number} or \defn{Fubini number}, and is asymptotic to $\tfrac{1}{2} n!\log(2)^{-n-1}$ (see \cite{OEIS} for details). Letting $Y=e_{\chi}[\cC(\bR)]$, we thus find that $\psi_Y(n)=[a(2n)]_q$. We have $\log(a(n)) \ge n \log{n}-Cn$ for some constant $C$, and so
\begin{displaymath}
\log(a(2n)) \ge 2n \log(2n) - 2Cn \ge n \log{n}
\end{displaymath}
for $n \gg 0$. We thus have
\begin{displaymath}
\phi_Y(n) \ge A\exp(B\exp(n\log{n})),
\end{displaymath}
for explicit positive real numbers $A$ and $B$, and $n \gg 0$.
\end{example}

We close with one final computation. Recall that if $\phi$ is an endomorphism in an $F$-linear pre-Tannakian category then it has a categorical trace $\ul{\tr}(\phi)$, which belongs to $F$. If $X$ is an object, its categorical dimension $\ul{\dim}(X)$ is defined as the categorical trace of the identity. In what follows, we put $\chi(0)=0$.

\begin{proposition}
Let $\phi \colon X \to X$ be a morphism in $\cE$, and let $e_{\chi} [\phi] \colon e_{\chi} [X] \to e_{\chi} [Y]$ be the induced morphism in $\cC$. Then $\ul{\tr}(e_{\chi} [\phi])=\chi(\ul{\tr}(\phi))$.
\end{proposition}

\begin{proof}
Let $\bone$ be the unit object in $\cE$. The composition
\begin{displaymath}
\xymatrix@C=4em{
\bone \ar[r] & X^{\vee} \otimes X \ar[r]^{\id \otimes \phi} \ar[r] & X^{\vee} \otimes X \ar[r] & \bone }
\end{displaymath}
is multiplication by the scalar $a=\ul{\tr}(\phi)$. We thus find that the composition
\begin{displaymath}
\xymatrix@C=4em{
[\bone] \ar[r] & [X]^{\vee} \uotimes [X] \ar[r]^{\id \otimes [\phi]} \ar[r] & [X]^{\vee} \uotimes [X] \ar[r] & [\bone] }
\end{displaymath}
is multiplication by $[a] \in k[\bF]$. Now apply the idempotent $e_{\chi}$. On the one hand, the composition is multiplication by $\ul{\tr}(e_{\chi} [\phi])$. On the other, it is $e_{\chi} [a] = \chi(a) e_{\chi}$; and note that $e_{\chi}$ here is really the identity on $e_{\chi} [\bone]$. The result follows.
\end{proof}

\begin{corollary}
$\ul{\dim}(e_{\chi} [X])=\chi(\ul{\dim}{X})$.
\end{corollary}

\end{document}